\documentclass[a4paper,11pt]{article}
\usepackage[utf8]{inputenc}
\usepackage[colorlinks=true, allcolors=blue]{hyperref}
\usepackage{amsmath,amsthm,amssymb}
\usepackage{datetime}
\usepackage{xspace}

\numberwithin{equation}{section}
\newdateformat{mydate}{\THEDAY \ \monthname[\THEMONTH] \THEYEAR}
\mydate

\newcommand{\updated}[1]{#1}

\theoremstyle{plain}
\newtheorem{theorem}{Theorem}[section]
\newtheorem{lemma}[theorem]{Lemma}
\newtheorem{proposition}[theorem]{Proposition}

\theoremstyle{definition}
\newtheorem{example}[theorem]{Example}
\newtheorem{remark}[theorem]{Remark}

\newcommand{\Erdos}{Erd\H{o}s\xspace}
\newcommand{\Karonski}{Karo\'{n}ski\xspace}
\newcommand{\Luczak}{{\L}uczak\xspace} 
\newcommand{\Renyi}{R{\'e}nyi\xspace}

\newcommand{\abs}[1]{{\lvert#1\rvert}}

\newcommand{\floor}[1]{\left\lfloor #1 \right\rfloor}

\newcommand{\lesim}{\lesssim}

\newcommand{\weq}{\ = \ }

\newcommand{\wle}{\ \le \ }
\newcommand{\wge}{\ \ge \ }

\newcommand{\wlesim}{\ \lesssim \ }

\newcommand{\wasymp}{\ \asymp \ }

\newcommand{\E}{\mathbb{E}}
\newcommand{\pr}{\mathbb{P}}

\newcommand{\cond}{\, | \,}
\newcommand{\Var}{\operatorname{Var}}

\newcommand{\Cov}{\operatorname{Cov}}

\newcommand{\cA}{\mathcal{A}}

\newcommand{\cD}{\mathcal{D}}

\newcommand{\cH}{\mathcal{H}}
\newcommand{\cI}{\mathcal{I}}

\newcommand{\tx}{\tilde x}

\newcommand{\wto}{\ \to \ }

\usepackage[lmargin=25mm,rmargin=25mm,top=25mm,bottom=30mm,paperwidth=210mm,paperheight=11in]{geometry}

\begin{document}

\title{Connectivity of random hypergraphs with a given hyperedge size distribution}
\author{Elmer Bergman and Lasse Leskelä \\[1ex]
\small Aalto University \\
\small Department of Mathematics and Systems Analysis \\
\small Otakaari 1, 02015 Espoo, Finland
}
\date{\today}

\maketitle

\begin{abstract}
This article discusses random hypergraphs with varying hyperedge sizes, admitting large hyperedges with size tending to infinity, and heavy-tailed limiting hyperedge size distributions. The main result describes a threshold for the random hypergraph to be connected with high probability, and shows that the average hyperedge size suffices to characterise connectivity under mild regularity assumptions. Especially, the connectivity threshold is in most cases insensitive to the shape and higher moments of the hyperedge size distribution. Similar results are also provided for related random intersection graph models.
\end{abstract}


\section{Introduction}

The probability of connectivity is one of the most classical problems in the study of random graphs. 
In their seminal article, \Erdos and \Renyi \cite{Erdos_Renyi_1959} proved that
a random graph $G_{nm}$ with $n$ nodes and $m=m_n$ edges, is with high probability (whp) connected when 
$2 \frac{m}{n} - \log n \to \infty$, and whp disconnected when $2 \frac{m}{n} - \log n \to -\infty$.
The random graph $G_{nm}$ is a special instance of a random hypergraph $H_{nmd}$ with $n$ nodes and $m$ hyperedges of size $d$, with $d=2$.  An extension of the above result to random hypergraphs with $d = O(1)$, noted in \cite{Coja-Oghlan_Moore_Sanwalani_2007} and formally proved by Poole \cite{Poole_2015}, states that $H_{nmd}$ is whp connected when
$d \frac{m}{n} - \log n \to \infty$ and whp disconnected when $d \frac{m}{n} - \log n \to -\infty$.
Instead of full connectivity, most earlier research on random hypergraphs has focused on sparse regimes critical for the emergence of a giant connected component (Behrisch, Coja-Oghlan, and Kang \cite{Behrisch_Coja-Oghlan_Kang_2010,Behrisch_Coja-Oghlan_Kang_2014}; \Karonski and \Luczak \cite{Karonski_Luczak_2002}; Schmidt-Pruzan and Shamir \cite{Schmidt-Pruzan_Shamir_1985}).  
On the other hand, fundamental statistical questions related to identifying parameters and detecting communities in networks \updated{with} multiple layers and higher-order interactions often require a quantitative understanding of the full connectivity threshold \cite{Ahn_Lee_Suh_2019,Chien_Lin_Wang_2019,Chodrow_Veldt_Benson_2021,Kaminski_Pralat_Theberge_2021}.
This motivates to study random hypergraphs where the hyperedge sizes may not be universally bounded, and the hyperedge sizes may be highly variable. Our main research objective is to study the effect of inhomogeneous hyperedge sizes on critical thresholds of connectivity in large hypergraphs.

\subsection{Notations}

\paragraph{Hypergraphs.}

A \emph{hypergraph} is a pair $h = (V_h, E_h)$ where $V_h$ is a finite set of objects called nodes and $E_h$ is a collection of nonempty subsets\footnote{\updated{In this article we restrict to hypergraphs not containing degenerate or parallel hyperedges. In contrast to \cite{Chodrow_2020}, we may therefore represent each hyperedge and the collection of hyperedges as sets instead of multisets.}}
 of $V_h$ called hyperedges.  A hypergraph $h$ is \emph{connected} if for each partition of $V_h$ into nonempty sets $V_1$ and $V_2$ there exists a hyperedge $e \in E_h$ such that $V_1 \cap e \ne \emptyset$ and $V_2 \cap e \ne \emptyset$.  A node of a hypergraph is \emph{isolated} if it is not contained in any hyperedge of size at least two. The \emph{empirical hyperedge size distribution} of a hypergraph $h$ is the probability mass function $f(x) = \abs{E_h}^{-1} \sum_{e \in E_h} \delta_{\abs{e}}(x)$ where $\delta_a(x)=1$ if $x=a$ and $\delta_a(x)=0$ otherwise. The \emph{2-section} of a hypergraph $h$ is the graph $h' = (V'_h, E'_n)$ where $V'_h = V_h$ and $E'_h$ is the set of unordered node pairs contained in at least one hyperedge of~$h$.

\paragraph{Asymptotics.}
Standard Bachmann--Landau notations are employed as follows: We write
$a_n \ll b_n$ and $a_n = o(b_n)$ when $a_n/\abs{b_n} \to 0$, $a_n \lesim b_n$ and $a_n = O(b_n)$ when $\limsup a_n/\abs{b_n} < \infty$, and $a_n \sim b_n$ when $a_n/b_n \to 1$.

\paragraph{Probability.}
The expectation of a random variable defined on a probability space equipped with probability measure $\pr = \pr_n$ is denoted by $\E = \E_n$.  For clarity, the scale parameter $n=1,2,\dots$ is mostly omitted from the notations. An event $\cA = \cA_n$ is said to occur \emph{with high probability (whp)} if $\pr_n(\cA_n) \to 1$ as $n \to \infty$.

\subsection{Outline}
The rest of the article is organised as follows.  Section~\ref{sec:Main} presents the main results and 
Section~\ref{sec:Discussion} discusses their relevance and relation to earlier research.
The proofs are in Sections \ref{sec:Isolated}--\ref{sec:Proofs}, with Section \ref{sec:Isolated} analysing isolated nodes, Section~\ref{sec:Connectivity} connectivity, and Section \ref{sec:Proofs} summarising details.

\section{Main results}
\label{sec:Main}

The main results are given separately for three different, but closely related network models. 
Section~\ref{sec:Hypergraph} discusses random hypergraphs with given hyperedge sizes, 
and Section~\ref{sec:RIG} random intersection graphs obtained as a 2-section of a random hypergraph.
Section~\ref{sec:SRH} describes a random hypergraph model which allows to sample instances of both earlier models.

\subsection{Random hypergraphs with given hyperedge sizes}
\label{sec:Hypergraph}

Fix integers $n,m \ge 1$ and a probability distribution $f$ on $\{1,\dots,n\}$.  
\updated{
Denote by $\cH_{nmf}$ the collection of hypergraphs on node set $\{1,\dots,n\}$ having
$m$ hyperedges and empirical hyperedge size distribution $f$.
We see that the set $\cH_{nmf}$ is nonempty when for all $x=1,\dots,n$, $mf(x)$ is an integer satisfying $m f(x) \le \binom{n}{x}$. In this case we let $H_{nmf}$ be a random hypergraph sampled uniformly at random from $\cH_{nmf}$.
Then $H_{nmf}$ contains $mf(1)$ hyperedges of size 1, $mf(2)$ hyperedges of size 2, and so on.  
}
Because hyperedges of size one do not affect connectivity, relevant moments of the hyperedge size distribution are denoted by
\begin{equation}
 \label{eq:fMoment}
 (f)_{r} \weq \sum_{x=2}^n x^r f(x).
\end{equation}

\begin{theorem}
\label{the:Hypergraph}
If $m=m_n$ and $f = f_n$ are such that
$\sum_{x=1}^n \binom{n}{x}^{-1} f(x)^2 \ll m^{-2}$ and
$(f)_2 \lesim (f)_0$, then
\[
 \pr(\text{$H_{nmf}$ is connected})
 \wto
 \begin{cases}
  0 &\quad \text{if $\log n - \frac{m}{n} (f)_{1} \to +\infty$}, \\
  1 &\quad \text{if $\log n - \frac{m}{n} (f)_{1} \to -\infty$}.
 \end{cases}
\]
\end{theorem}

Let $H_{nmd}$ be a random hypergraph sampled uniformly at random from the set of hypergraphs on node set $\{1,\dots,n\}$ having $m$ hyperedges of size $d$. Such a hypergraph is a special instance of $H_{nmf}$ where $f$ is the Dirac point mass at $d$, and therefore Theorem~\ref{the:Hypergraph} characterises the connectivity of $H_{nmd}$ when $m \ll \binom{n}{d}^{1/2}$ and $d = O(1)$. In this special case where all hyperedges are of equal size $d$, Theorem~\ref{the:HypergraphRegular} below provides an alternative characterisation of connectivity which admits large hyperedges not constrained by the assumption that $d = O(1)$.

\begin{theorem}
\label{the:HypergraphRegular}
If $m=m_n$ and $d=d_n$ are such that 
$m \ll \binom{n}{d}^{1/2}$ and $2 \le d \le n$, then
\[
 \pr(\text{$H_{nmd}$ is connected})
 \wto
 \begin{cases}
  0 &\quad \text{if $\log n + m \log ( 1 - d/n ) \to +\infty$}, \\
  1 &\quad \text{if $\log n + m \log ( 1 - d/n ) \to -\infty$}.
 \end{cases}
\]
\end{theorem}

\subsection{Random intersection graphs}
\label{sec:RIG}

Fix integers $n,m \ge 1$ and a probability distribution $f$ on $\{0,\dots,n\}$. Let $G_{nmf}$ be a random graph on node set $\{1,\dots,n\}$ generated by first sampling mutually independent random sets $V_1,\dots,V_m \subset \{1,\dots,n\}$ from probability mass function $A \mapsto f(\abs{A}) \binom{n}{\abs{A}}^{-1}$, and then declaring each unordered node pair $\{i,j\}$ adjacent if there exist at least one set $V_k$ containing both $i$ and $j$. This model is known as the passive random intersection graph \cite{Godehardt_Jaworski_2001}.

\begin{theorem}
\label{the:RIG}
For any $m=m_n$ and $f = f_n$ such that $m \ge 1$ and $(f)_2 \lesim (f)_0$,
\[
 \pr(\text{$G_{nmf}$ is connected})
 \wto
 \begin{cases}
  0 &\quad \text{if $\log n - \frac{m}{n} (f)_{1} \to +\infty$}, \\
  1 &\quad \text{if $\log n - \frac{m}{n} (f)_{1} \to -\infty$}.
 \end{cases}
\]
\end{theorem}

Let $G_{nmd}$ be a special instance of the above random intersection graph model in which $f=\delta_d$ is the Dirac point mass at $d$.
\begin{theorem}
\label{the:RIGRegular}
For any $m=m_n$ and $d = d_n$ such that $m \ge 1$ and $2 \le d \le n$,
\[
 \pr(\text{$G_{nmd}$ is connected})
 \wto
 \begin{cases}
  0 &\quad \text{if $\log n + m \log ( 1 - d/n ) \to +\infty$}, \\
  1 &\quad \text{if $\log n + m \log ( 1 - d/n ) \to -\infty$}.
 \end{cases}
\]
\end{theorem}

\subsection{Shotgun random hypergraphs}
\label{sec:SRH}

The results in Sections~\ref{sec:Hypergraph}--\ref{sec:RIG} will be proved by analysing a more general random hypergraph model defined as follows. Fix a probability distribution $F = f^{(1)} \times \cdots \times f^{(m)}$ on
$\{0,\dots,n\}^m$, and consider a random hypergraph $H^*_{nmF}$ with node set $\{1,\dots,n\}$ and hyperedge set\footnote{\updated{The cardinality of the hyperedge set may be less than $m$ due to possible duplicates among $V_1,\dots,V_m$.}} $\{V_1,\dots,V_m\}$, where the sets $V_1,\dots, V_m \subset \{1,\dots,n\}$ are mutually independent and such that the probability mass function of $V_k$ equals $A \mapsto  f^{(k)}(\abs{A}) \binom{n}{\abs{A}}^{-1}$.
The 2-section of $H^*_{nmF}$ corresponds to a generalisation of the passive random intersection graph in \cite{Godehardt_Jaworski_2001,Godehardt_Jaworski_Rybarczyk_2007}, and a special case of a Bernoulli superposition graph model in \cite{Bloznelis_Karjalainen_Leskela_2022,Bloznelis_Karjalainen_Leskela_2023,Bloznelis_Leskela_2023,Grohn_Karjalainen_Leskela_2024,Karjalainen_VanLeeuwaarden_Leskela_2018,Petti_Vempala_2022}.
Denote
\begin{equation}
 \label{eq:FMoment}
 (F)_{r} \weq \frac{1}{m} \sum_{k=1}^{m} \sum_{x=2}^n x^r f^{(k)}(x).
\end{equation}

\begin{theorem}
\label{the:SRH}
For any $m=m_n$ and $F = F_n$ such that $(F)_2 \lesim (F)_0$,
\[
 \pr(\text{$H^*_{nmF}$ is connected})
 \wto
 \begin{cases}
  0 &\quad \text{if $\log n - \frac{m}{n} (F)_{1} \to +\infty$}, \\
  1 &\quad \text{if $\log n - \frac{m}{n} (F)_{1} \to -\infty$}.
 \end{cases}
\]
\end{theorem}

Denote by $H^*_{nmd}$ a special instance of $H^*_{nmF}$ with $F$ being equal to the Dirac point mass at $(d,\dots, d) \in \{0,\dots, n\}^m$.

\begin{theorem}
\label{the:SRHRegular}
For any $m=m_n$ and $d=d_n$ such that $m \ge 1$ and $2 \le d \le n$,
\[
 \pr(\text{$H^*_{nmd}$ is connected})
 \wto
 \begin{cases}
  0 &\quad \text{if $\log n + m \log ( 1 - d/n ) \to +\infty$}, \\
  1 &\quad \text{if $\log n + m \log ( 1 - d/n ) \to -\infty$}.
 \end{cases}
\]
\end{theorem}

The expected number of isolated nodes in $H^*_{nmd}$ equals $e^{\lambda}$ with
$\lambda = \log n + m \log ( 1 - d/n )$ (see Proposition~\ref{the:MeanIsolated}).
Therefore, Theorem \ref{the:SRHRegular} implies that the threshold for connectivity is the same as the threshold for the existence of isolated nodes in $H^*_{nmd}$.

\section{Discussion}
\label{sec:Discussion}

\subsection{Role of large hyperedges}

For uniform random hypergraphs, Poole \cite{Poole_2015} has shown that  if $d = O(1)$, then 
\[
 \pr(\text{$H_{nmd}$ is connected})
 \wto
 \begin{cases}
  0 &\quad \text{if $\log n - \frac{m}{n}d \to +\infty$}, \\
  1 &\quad \text{if $\log n - \frac{m}{n}d \to -\infty$}.
 \end{cases}
\]
This characterisation follows as a special instance of Theorem~\ref{the:HypergraphRegular},
because when $d = O(1)$, we may replace $\lambda = \log n + m \log(1-d/n)$ in Theorem~\ref{the:HypergraphRegular} by $\mu = \log n - \frac{m}{n}d$ (see Remark~\ref{rem:LargeHyperedges}). 
On the other hand, in regimes with $\frac{n}{\log n} \ll d \ll n$, property $\mu \to \infty$ does not in general imply that $\lambda \to \infty$, and it may happen that $H_{nmd}$ is whp connected even though $\log n - \frac{m}{n}d \to \infty$, see Example~\ref{exa:LargeHyperedges} below.

\begin{remark}
\label{rem:LargeHyperedges}
When hyperedge sizes are bounded by $d \lesim \frac{n}{\log n}$, then we may replace $\lambda = \log n + m \log(1-d/n)$ in Theorem~\ref{the:HypergraphRegular} by $\mu = \log n - \frac{m}{n}d$. To see why this is true, we note that (i) $\mu \to -\infty$ together with $\lambda \le \mu$ implies that $\lambda \to -\infty$. We also note that (ii) $\mu \to \infty$ implies that $m \lesim \frac{n}{d}\log n$, and together with $d \lesim \frac{n}{\log n}$ we find that (see Lemma~\ref{the:LogTaylor})
$0 \le \mu-\lambda \le m \frac{(d/n)^2}{1-d/n} \asymp m (d/n)^2
\lesim \frac{n}{d} (d/n)^2 \log n  
= (d/n) \log n 
= O(1)$.
\end{remark}

\begin{example}
\label{exa:LargeHyperedges}
Assume that $\log n \ll m \ll \log^2 n$ and
$d = \lfloor \frac{n}{m}( \log n - \omega ) \rfloor$ for some $1 \ll \omega \ll \frac{\log^2 n}{m}$. For example,
$m = \lfloor \log^{3/2} n \rfloor$ and $\omega = \frac{\log n}{m^{1/2}}$. Denoting
$\lambda = \log n \updated{+} m\log(1-d/n)$ and $\mu = \log n - \frac{m}{n}d$, we may verify (see below) that
\begin{equation}
 \label{eq:LargeHyperedges}
 \lambda \to \updated{-\infty},
 \qquad \mu \to \updated{+\infty},
 \qquad m \ll \binom{n}{d}^{1/2}.
\end{equation}
Theorems \ref{the:HypergraphRegular} and~\ref{the:SRHRegular} then tell that $H_{nmd}$ and $H^*_{nmd}$ are whp connected even though $\log n - \frac{m}{n}d \to \infty$.

To verify the limits in \eqref{eq:LargeHyperedges}, denote 
$d' = \frac{n}{m}( \log n - \omega )$. Then $d' \sim \frac{n}{m} \log n$ is bounded by $\frac{n}{\log n} \ll d' \ll n$. Therefore, $d = \floor{d'}$ also satisfies $d \sim \frac{n}{m} \log n$. Then $\mu = \log n - \frac{m}{n} d \ge \log n - \frac{m}{n} d' = \omega$ implies that $\mu \to \infty$. On the other hand, the inequality $\log(1-t) \le -t - \frac12 t^2$ for $t \ge 0$ implies that
\begin{align*}
 \lambda
 &\wle \log n - \frac{m}{n}(d'-1) - \frac12 m (d/n)^2 \\
 &\wle \omega + \frac{m}{n} - \frac12 m (d/n)^2,
\end{align*}
so that $\lambda \le \omega + O(1) - \frac12(1+o(1)) \frac{\log^2 n}{m}$, and we conclude that $\lambda \to -\infty$. We also note that $d \ll n$ implies that $\log(n/d) \gg 1$.
Therefore, $\binom{n}{d} \ge (\frac{n}{d})^d$ implies that
$\log \binom{n}{d} \ge d \log(n/d) \gg d \sim \frac{n}{m}\log n \gg \frac{n}{\log n}$.
Because $\log m \le 2 \log\log n - \omega(1)$, we see that $\log \binom{n}{d} - 2 \log m \to \infty$,
and hence $m \ll \binom{n}{d}^{1/2}$.
\end{example}

Similarly, Godehardt, Jaworski, and Rybarczyk \cite{Godehardt_Jaworski_Rybarczyk_2007} showed that for a random intersection graph $G_{nmf}$ with $f=f_n$ having support universally bounded by $\max\{x: f(x) \ne 0\} = O(1)$, $G_{nmf}$ is whp connected when $\log - \frac{m}{n} (f)_1 \to - \infty$ 
and whp disconnected when $ \log n - \frac{m}{n} (f)_1 \to +\infty$. Theorem~\ref{the:RIG} extends this characterisation to random intersection graphs where the support of $f$ is not universally bounded.
Theorem~\ref{the:RIGRegular} provides a sharper characterisation for degenerate distributions which is relevant in cases with $d \gg \frac{n}{\log n}$.

\subsection{Role of variable hyperedge sizes}

Example~\ref{exa:SmallAndFull} below shows that in some extremal cases, $H^*_{nmF}$ have the expected number of isolated nodes converging to infinity, also implying that $\log n - \frac{m}{n} (F)_{1} \to +\infty$, while still producing graph samples which are whp connected.  Therefore, some regularity conditions, such as the condition $(F)_2 \lesim (F)_0$, must be imposed for the characterisation in Theorem~\ref{the:SRH} to be valid.

\begin{example}
\label{exa:SmallAndFull}
Consider a random hypergraph $H^*_{nmF}$ with $F = f \times \cdots \times f$, where $f = (1-p)\delta_2 + p\delta_n$, $p = 1 - n^{-1/(2m)}$ and $m \le \frac16 n \log n$.
In this case $H^*_{nmF}$ contains a full hyperedge $[n]$ with probability $1-(1-p)^m = 1 - n^{-1/2}$, and therefore
\[
 \pr(\text{$H^*_{nmF}$ connected})
 \wge 1 - n^{-1/2}.
\]
In this case any particular node is isolated with probability $P_1 = ((1-p)(1-2/n))^m$.
We also note (Lemma~\ref{the:LogTaylor}) that
$
 \log(1-2/n)
 \ge - \frac{2/n}{1-2/n}
 \ge - 2(1+o(1)) n^{-1},
$
so that $(1-2/n)^m \ge n^{-1/3 - o(1)}$. Hence the expected number of isolated nodes equals
$
 n P_1
 = n^{1/2} (1-2/n)^m
 \ge n^{1/6-o(1)}.
$
We conclude that $H^*_{nmF}$ is whp connected, even though the expected number of isolated nodes is at least $n^{1/6-o(1)}$. The same conclusions are valid for the random intersection graph $G_{nmf}$ obtained as the 2-section of $H^*_{nmF}$.
\end{example}

\subsection{Conclusion}

This article's main findings indicate that under mild regularity assumptions, the average hyperedge size suffices to characterise connectivity, and other characteristics of the hyperedge size distribution, such as higher moments or heavy tails, are not relevant to connectivity.  This finding is in line with similar studies related to the giant component in random graphs (see Deijfen, Rosengren, and Trapman \cite{Deijfen_Rosengren_Trapman_2018}; Leskelä and Ngo \cite{Leskela_Ngo_2017}). The analysis in the present article was restricted to hyperedge size distributions with a universally bounded second moment.  Whether or not this regularity condition may be relaxed remains a problem of further research.

\section{Isolated nodes}
\label{sec:Isolated}

In this section we will analyse the number of isolated nodes in the shotgun random hypergraph $H^*_{nmF}$ defined in Section~\ref{sec:SRH}. Recall that a node of a hypergraph is called \emph{isolated} if it is not contained in any hyperedge of size at least two.
\updated{
Section~\ref{sec:ExistenceIsolatedNodes} summarises our findings
in Propositions~\ref{the:MeanIsolated}--\ref{the:YesIsolatedRIG}.
Section~\ref{sec:IsolationProbabilities} presents results for isolation probabilities that used
in Section~\ref{sec:ProofsIsolatedRIG} to prove Propositions~\ref{the:MeanIsolated}--\ref{the:YesIsolatedRIG}.
}
\updated{
\subsection{Existence of isolated nodes}
\label{sec:ExistenceIsolatedNodes}
}

\begin{proposition}
\label{the:MeanIsolated}
The expected number of isolated nodes in $H^*_{nmF}$ equals $\exp(\lambda)$ where 
\begin{equation}
 \label{eq:MeanIsolated}
 \lambda
 \weq \log n + \sum_{k=1}^m \log\left( 1 - \frac{1}{n} \sum_{x=2}^n x f^{(k)}(x) \right).
\end{equation}
\end{proposition}

Proposition~\ref{the:MeanIsolated} shows that 
the expected number of isolated nodes tends to zero when $\lambda \to -\infty$ and to infinity when $\lambda \to +\infty$. The following result tells that this characterisation may be simplified under a mild regularity condition.

\begin{proposition}
\label{the:ExpectedIsolatedNewer}
Define $\lambda$ by \eqref{eq:MeanIsolated}, and let $\mu = \log n - \frac{m}{n} (F)_1$ with the moments $(F)_r$ defined by \eqref{eq:FMoment}. 
Then $\lambda \le \mu$. Moreover, if $(F)_1 \lesim m^{-1} n \log n$ and $(F)_2 \lesim (F)_0$, 
then \updated{$\lambda = \mu - O(\frac{\log n}{n})$}.
\end{proposition}

Proposition~\ref{the:MeanIsolated} combined with Markov's inequality implies that all nodes in $H^*_{nmF}$ are nonisolated whp when $\lambda \to -\infty$.  The converse is not true in general, as Example~\ref{exa:SmallAndFull} confirms that $\lambda \to +\infty$ does not in general imply that $H^*_{nmF}$ contains an isolated node whp.  Therefore, some extra assumptions are needed.
Proposition~\ref{the:YesIsolatedRIG} provides such regularity condition. The proofs of these propositions are postponed to Section~\ref{sec:ProofsIsolatedRIG}.

\begin{proposition}
\label{the:YesIsolatedRIG}
For $\lambda$ defined by \eqref{eq:MeanIsolated},
\[
 1-e^{-\lambda} \wle
 \pr( \text{$H^*_{nmF}$ contains no isolated nodes} )
 \wle e^{-\lambda} + \exp\left( \sum_{k=1}^m \frac{\Var(Z_k)}{(\E Z_k)^2} \right) - 1,
\]
where $Z_k = 1 - n^{-1}X_k 1(X_k \ge 2)$, and $X_k$ is an \updated{$f^{(k)}$}-distributed random integer.
\end{proposition}

\subsection{Analysis of isolation probabilities}
\label{sec:IsolationProbabilities}

\updated{
Denote by $P_1$ the probability that any particular node is isolated in $H_{nmf}^*$,
and by $P_2$ the probability that any particular two nodes are both isolated in $H_{nmf}^*$.
We will next derive formulas and bounds for $P_1$ and $P_2$ in terms of functions
}
\begin{align}
 \label{eq:phi1}
 \phi_1(x) &\weq 1 - \frac{x 1(x \ge 2)}{n}, \\
 \label{eq:phi2}
 \phi_2(x) &\weq \left(1-\frac{x1(x \ge 2)}{n} \right) \left(1-\frac{x1(x \ge 2)}{n-1} \right).
\end{align}

\begin{lemma}
\label{the:IsolationProbabilityRIG}
The probability that any particular node is isolated in $H^*_{nmF}$ is given by
$P_1 = \prod_{k=1}^m (\sum_{x=0}^n \phi_1(x) f^{(k)}(x))$, and the
probability that any particular two nodes both are isolated in $H^*_{nmF}$ equals
$P_2 = \prod_{k=1}^m (\sum_{x=0}^n \phi_2(x) f^{(k)}(x))$ with
$\phi_1$ defined by \eqref{eq:phi1} and
$\phi_2$ defined by \eqref{eq:phi2}.
\end{lemma}
\begin{proof}
We may write $H^*_{nmF}$ as a union $\cup_{k=1}^m H^*_k$ where $H^*_k$ is a hypergraph on $\{1,\dots,n\}$ having $V_k$ as the only hyperedge. We note that a node $i$ is isolated in $H^*_{nmF}$ if and only if it isolated in $H^*_k$ for all $k$.  Because $H^*_1,\dots,H^*_m$ are mutually independent, we see that $P_1 = \prod_{k=1}^m \E \phi_1(X_k)$ where $X_k = \abs{V_k}$ and $\phi_1(x)$ is the conditional probability of node~$i$ being isolated in $H^*_k$ given that $\abs{V_k} = x$. Then $\phi_1(x) = 1 - x/n$ for $x \ge 2$ and $\phi_1(x)=1$ otherwise.

Similarly, the probability that any distinct nodes $i$ and $j$ both are isolated in $H^*_{nmF}$ equals $P_{2} = \prod_{k=1}^m \E \phi_2(X_k)$ where $\phi_2(x)$ is the conditional probability that $i$ and $j$ both are isolated in $H^*_k$ given that $\abs{V_k} = x$. Then $\phi_2(x) = 1$ for $x = 0,1$. For $x \ge 2$, nodes $i$ and $j$ both are isolated in $H^*_k$ if and only if neither of these nodes is contained in $V_k$, and given $\abs{V_k} = x$, this occurs with probability $(1 - \frac{x}{n}) (1 - \frac{x}{n-1})$.
\end{proof}

\begin{lemma}
\label{the:PairIsolationRatioRIG}
The pair isolation probability ratio $P_2/P_1^2$ is bounded by
\[
 \log \frac{P_2}{P_1^2}
 \wle \sum_{k=1}^m \frac{\Var(\phi_1(X_k))}{(\E \phi_1(X_k))^2}.
\]
where $X_k$ is distributed according to $f^{(k)}$ and $\phi_1$ is defined by \eqref{eq:phi1}.
\end{lemma}

\begin{proof}
Denote by $I_i$ (resp.\ $I_{ik}$) the indicator variable of the event that node $i$ is isolated in $H^*_{nmF}$ (resp.\ layer~$H^*_k$).  Because $I_i = \prod_{k=1}^m I_{ik}$ and the sets $V_k$ are mutually independent, it follows that
$P_1 = \E I_1 = \prod_{k=1}^m \E I_{1k}$
and 
$P_2 = \E(I_1 I_2) = \prod_{k=1}^m \E I_{1k} I_{2k}$.
Then
\[
 \frac{P_2}{P_1^2}
 \weq \prod_{k=1}^m \frac{\E I_{1k} I_{2k}}{(\E I_{1k})^2},
\]
so that
\begin{equation}
 \label{eq:LogPairRatioBound1}
 \log \frac{P_2}{P_1^2}
 \weq \sum_{k=1}^m \log \frac{\E I_{1k}I_{2k}}{(\E I_{1k})^2}
 \wle \sum_{k=1}^m \left( \frac{\E I_{1k}I_{2k}}{(\E I_{1k})^2} - 1 \right)
 \weq \sum_{k=1}^m \frac{\Cov( I_{1k}, I_{2k})}{(\E I_{1k})^2}.
\end{equation}

To analyse the covariance term in \eqref{eq:LogPairRatioBound1}, denote $X_k = \abs{V_k}$.  Then by conditioning on $X_k$, we find that
\[
 \Cov( I_{1k}, I_{2k})
 \weq \Cov( \E(I_{1k} | X_k), \E(I_{2k} | X_k ) ) + \E \Cov( I_{1k}, I_{2k} | X_k ).
\]
We also note that
$\E(I_{1k} | X_k ) = \E(I_{2k} | X_k ) = \phi_1(X_k)$, so that
\[
 \Cov( \E(I_{11} | X_k ), \E(I_{21} | X_k ) )
 \weq \Var(\phi_1(X_k)).
\]
Furthermore,
\[
 \Cov( I_{1k}, I_{2k} | X_k)
 \weq \phi_2(X_k) - \phi_1^2(X_k).
\]
We conclude that
\[
 \Cov( I_{1k}, I_{2k})
 \weq \Var(\phi_1(X_k)) + \E( \phi_2(X_k) - \phi_1^2(X_k)),
\]
where $\phi_1$ and $\phi_2$ are given by \eqref{eq:phi1} and \eqref{eq:phi2}. For $x < 2$ we see that $\phi_2(x) - \phi_1^2(x) = 0$. For $x \ge 2$, we see that
\[
 \phi_2(x) - \phi_1^2(x)
 \weq \Big(1- \frac{x}{n} \Big) \Big(1- \frac{x}{n-1} \Big) - \Big(1- \frac{x}{n} \Big)^2
 \weq - \frac{1}{n-1} \Big(1- \frac{x}{n} \Big) \frac{x}{n}.
\]
Therefore, $\phi_2(x) - \phi_1^2(x) \le 0$ for all $x$. 
Hence $\Cov( I_{1k}, I_{2k}) \le \Var(\phi_1(X_k))$.  By substituting this bound into \eqref{eq:LogPairRatioBound1}, and noting that $\E I_{1k} = \E \phi_1(X_k)$, we find that
\[
 \log \frac{P_2}{P_1^2}
 \wle \sum_{k=1}^m \frac{\Var(\phi_1(X_k))}{(\E \phi_1(X_k))^2}.
\]
\end{proof}

\updated{
\subsection{Proofs of Propositions \ref{the:MeanIsolated}--\ref{the:YesIsolatedRIG}}
\label{sec:ProofsIsolatedRIG}
}

\begin{proof}[Proof of Proposition \ref{the:MeanIsolated}]
The expected number of isolated nodes equals $n P_1$ where $P_1$ is the probability that any particular node is isolated in $H^*_{nmF}$.  The claim follows after noting that by Lemma~\ref{the:IsolationProbabilityRIG},
$
 P_1
 = \prod_{k=1}^m \sum_{x=0}^n \phi_1(x) f^{(k)}(x) 
 = \prod_{k=1}^m \left(1 - \frac{1}{n} \sum_{x=2}^n x f^{(k)}(x) \right),
$
where $\phi_1$ is defined by \eqref{eq:phi1}.
\end{proof}

\begin{proof}[Proof of Proposition \ref{the:ExpectedIsolatedNewer}]
The inequality $\log(1-t) \le -t$ implies that
\updated{
\begin{align*}
 \lambda - \log n
 \weq \sum_{k=1}^m \log\left( 1 - \frac{1}{n} \sum_{x=2}^n x f^{(k)}(x) \right) 
 \wle - \sum_{k=1}^m \frac{1}{n} \sum_{x=2}^n x f^{(k)}(x).
\end{align*}
By noting that the right side of the above inequality equals $-\frac{m}{n} (F_1)$ with
$(F)_1$ defined by \eqref{eq:FMoment},
we conclude that 
$\lambda \le \mu$ for $\mu = \log n - \frac{m}{n} (F)_1$.
}

Assume next that $(F)_2 \lesim (F)_0$.
Because $(F)_1 \ge 2 (F)_0$ and $(F)_2 \lesim (F)_0$, it follows that $(F)_2 \lesim (F)_1$.
Denote $a = \max_k \tx_k$ where $\tx_k = (f^{(k)})_1$, and observe that
$(f^{(k)})_1^2 \le (f^{(k)})_2$. Therefore,
$
 a^2
 \le \sum_{k=1}^m (f^{(k)})_2
 = m (F)_2.
$
The assumption $(F)_1 \lesim m^{-1} n \log n $ now implies that
$a \lesim (n \log n)^{1/2}$.
Especially, $\max_k \tx_k/n \to 0$.
We may then apply the bound $\log(1-t) \ge -\frac{t}{1-t}$, $0 \le t < 1$, to notice that
\[ 
 \lambda
 \weq \log n + \sum_{k=1}^m \log ( 1 - \tx_k/n )
 \wge \log n - \sum_{k=1}^m \frac{\tx_k/n}{1 - \tx_k/n}.
\]
By writing $\mu = \log n - \sum_{k=1}^m \tx_k/n$, we find that
\begin{align*}
 \mu - \lambda
 \wle \sum_{k=1}^m \frac{\tx_k/n}{1 - \tx_k/n} - \sum_{k=1}^m \tx_k/n
 \weq \frac{1}{n^2} \sum_{k=1}^m \frac{\tx_k^2}{1 - \tx_k/n}.
\end{align*}
Because $\tx_k^2 \le (f^{(k)})_2$, it follows that
\begin{align*}
 \mu - \lambda
 \wle \frac{1}{n^2} \sum_{k=1}^m \frac{(f^{(k)})_2}{1-a/n} 
 \weq \frac{m}{n^2} \frac{(F)_2}{1-a/n}
 \wasymp \frac{m}{n^2} (F)_2
 \wlesim \frac{m}{n^2} (F)_1
 \wlesim \frac{\log n}{n}.
\end{align*}
\end{proof}

\begin{proof}[Proof of Proposition~\ref{the:YesIsolatedRIG}]
\updated{ The expected number of isolated nodes equals} $\E I = n P_1$ where $P_1$ is the probability that any particular node is isolated in $G$.  
Lemma~\ref{the:IsolationProbabilityRIG} implies that
$P_1 = \prod_{k=1}^m (\sum_{x=0}^n \phi_1(x) f^{(k)}(x))$ with $\phi_1$ defined by \eqref{eq:phi1}.
It follows that $\E I = e^\lambda$. Markov's inequality $\pr(I \ge 1) \le \E I$ then implies that $\pr(I \ne 0) \le e^\lambda$, and the claim follows.

Next, Chebyshev's inequality implies that
$\pr( I = 0 ) \le \pr( \abs{I - \E I} \ge \E I) \le \frac{\Var(I)}{(\E I)^2}$.
Observe that $\E I = n P_1$ and $\E I^2 = n P_1 + n(n-1) P_2$ where $P_1$ and $P_2$
are given by Lemma~\ref{the:IsolationProbabilityRIG}.
Hence
$
 \Var(I)
 = n P_1 + n(n-1) P_2 - n^2 P_1^2,
$
and we conclude that
\[
 \pr(I = 0)
 \wle \frac{1}{\E I} + (1-1/n) \frac{P_{2}}{P_1^2} - 1
 \wle \frac{1}{\E I} + \frac{P_{2}}{P_1^2} - 1.
\]
By Lemma~\ref{the:PairIsolationRatioRIG}, 
$ \log \frac{P_2}{P_1^2} \le \sum_{k=1}^m \frac{\Var(\phi_1(X_k))}{(\E \phi_1(X_k))^2}$,
and it follows that
\[
 \pr(I = 0)
 \wle \frac{1}{\E I} + \exp\left( \sum_{k=1}^m \frac{\Var(\phi_1(X_k))}{(\E \phi_1(X_k))^2} \right) - 1.
\]
\end{proof}

\section{Proving connectivity}
\label{sec:Connectivity}
\subsection{Upper bounds on cut probabilities}

Fix integers $1 \le r \le n/2$ and $1 \le x \le n$.
Denote by $q_r(x)$ the probability that a set of size $x$ sampled uniformly at random from $[n]$ is either fully contained in $[r]$ or fully contained in $[n] \setminus [r]$.

\begin{lemma}
\label{the:CutBoundDet}
Fix an integer $1 \le r \le n/2$. For any $2 \le x \le n$,
\begin{equation}
 \label{eq:ShotGunBound}
 q_r(x)
 \wle \Big( 1 + \frac{r^x}{(n-r)^x} \Big) \Big( 1 - \frac{x}{n} \Big)^r,
\end{equation}
and
\begin{equation}
 \label{eq:PairBound}
 q_r(x)
 \wle 1 - 2 \frac{r}{n} \Big( 1 - \frac{r}{n} \Big).
\end{equation}
Furthermore, for any random variable $X$ with values in $\{0,\dots,n\}$,
\begin{equation}
 \label{eq:ManyIsolatedBoundNew1}
 \E q_r(X)
 \wle 1 - \frac{r}{n} \E X 1(X \ge 2)
 + \frac{r^2}{(n-r)^2} \pr(X \ge 2)
 + \frac12 \frac{r^2}{n^2} \E X^2 1(X \ge 2)
\end{equation}
and
\begin{equation}
 \label{eq:ManyIsolatedBoundNew2}
 \E q_r(X)
 \wle \pr(X < 2) + e^{- 2 \frac{r}{n}( 1 - \frac{r}{n}) } \pr(X \ge 2).
\end{equation}

\end{lemma}
\begin{proof}
Fix an integer $1 \le r \le n/2$.
A simple counting argument shows that
\[
 q_r(x)
 \weq 
 \begin{cases}
  1, &\quad x \le 1, \\
  \frac{\binom{r}{x}}{\binom{n}{x}} + \frac{\binom{n-r}{x}}{\binom{n}{x}}, &\quad 1 < x \le r, \\
  \frac{\binom{n-r}{x}}{\binom{n}{x}}, &\quad r < x \le n-r, \\
  0, &\quad x > n-r.
 \end{cases}
\]
By applying a shotgun lemma (Lemma~\ref{the:Shotgun}), we see that
${ \binom{n-r}{x} } / { \binom{n}{x} } \le (1-x/n)^r$ for all
$x \le n-r$, and
${ \binom{r}{x} } / { \binom{n-r}{x} } \le ( 1 - \frac{n-2r}{n-r})^x = (\frac{r}{n-r})^x$ for all
$x \le n$. As a consequence,
\begin{align*}
 \frac{\binom{r}{x}}{\binom{n}{x}} + \frac{\binom{n-r}{x}}{\binom{n}{x}}
 \wle \Big( 1 + \frac{r^x}{(n-r)^x} \Big) \Big( 1 - \frac{x}{n} \Big)^r
\end{align*}
for all $2 \le x \le r$.  Hence \eqref{eq:ShotGunBound} is valid for all $2 \le x \le r$, and also for all $r < x \le n-r$. Of course, \eqref{eq:ShotGunBound} is trivially true for $x > n-r$ as well.

To verify the second bound, we note that $x \mapsto q_r(x)$ is nonincreasing for $x \ge 2$, so that
$q_r(x) \le q_r(2)$ for all $x \ge 2$.  A direct counting argument shows that $q_r(2) = 1 - \frac{r(n-r)}{\binom{n}{2}}$. Therefore,
\[
 q_r(x)
 \wle 1 - \frac{r(n-r)}{\binom{n}{2}}
 \wle 1 - 2 \frac{r(n-r)}{n^2},
\]
and this confirms \eqref{eq:PairBound}.

Next, by \eqref{eq:ShotGunBound} we find that for all $2 \le x \le n$,
\begin{align*}
 q_r(x)
 \wle \left( 1 + \Big(\frac{r}{n-r} \Big)^x \right) (1-x/n)^r
 \wle \exp \left( \Big(\frac{r}{n-r} \Big)^x - \frac{rx}{n} \right).
\end{align*}
We find that
\begin{equation}
 \label{eq:QBoundRandom1}
 \E q_r(X)
 \wle \pr(X<2) + \E e^{-Z_r} 1(X \ge 2),
\end{equation}
where $Z_r = \frac{r}{n} X - (\frac{r}{n-r})^X$.  Because $r \mapsto \frac{r}{(n-r)^2}$ is increasing on $[1, n/2]$, we see that on the event $X \ge 2$, $Z_r \ge 2 \frac{r}{n} - (\frac{r}{n-r})^2 = \frac{r}{n} ( 2 - \frac{n r}{(n-r)^2} )
\ge \frac{r}{n} ( 2 - \frac{n^2/2}{(n/2)^2} )
= 0$. Hence $Z_r \ge 0$ on the event $X \ge 2$.
Because $e^{-t} \le 1 - t + \frac12 t^2$ for all $t \ge 0$, it follows that 
\[
 \E e^{-Z_r} 1(X \ge 2)
 \wle \E \left( 1 - Z_r + \frac12 Z_r^2 \right) 1(X \ge 2),
\]
and hence by \eqref{eq:QBoundRandom1},
\begin{align*}
 \E q_r(X)
 &\wle 1 - \E Z_r 1(X \ge 2) + \frac12 \E Z_r^2 1(X \ge 2) \\
 &\weq 1 - \frac{r}{n} \E X 1(X \ge 2) + \E \Big(\frac{r}{n-r} \Big)^X 1(X \ge 2) + \frac12 \E Z_r^2 1(X \ge 2) \\
 &\wle 1 - \frac{r}{n} \E X 1(X \ge 2) + \Big(\frac{r}{n-r} \Big)^2 \pr(X \ge 2) + \frac12 \frac{r^2}{n^2} \E X^2 1(X \ge 2).
\end{align*}
This confirms \eqref{eq:ManyIsolatedBoundNew1}.

To verify \eqref{eq:ManyIsolatedBoundNew2}, we note that \eqref{eq:PairBound} implies that
$
 q_r(x)
 \le 1 - 2 \frac{r}{n} \Big( 1 - \frac{r}{n} \Big)
 \le e^{- 2 \frac{r}{n}( 1 - \frac{r}{n}) }.
$
Therefore,
\[
 \E q_r(X)
 \weq \pr(X < 2) + \E q_r(X) 1(X \ge 2)
 \wle \pr(X < 2) + e^{- 2 \frac{r}{n}( 1 - \frac{r}{n}) } \pr(X \ge 2).
\]
\end{proof}

\subsection{Connectivity for constant layer sizes}

\begin{proposition}
\label{the:ConnectivityConstantSizes}
Fix integers $2 \le d \le n$ and $m \ge 1$.  Then
\[
 \pr(\text{$H^*_{nmd}$ is disconnected})
 \wle e^\lambda + \sum_{r=1}^\infty e^{(\lambda+5)r} + (2/e)^n,
\]
where $\lambda = \log n + m \log(1-d/n)$.
\end{proposition}
\begin{proof}
Denote by $\cD$ be the event that $H^*_{nmd}$ is disconnected, and by $\cI$ the event that $H^*_{nmd}$ contains isolated nodes. Recall that by Proposition~\ref{the:YesIsolatedRIG},
\begin{equation}
 \label{eq:Juhannus00}
 \pr(\cI)
 \wle e^\lambda.
\end{equation}
On the event $\cD \setminus \cI$ there exists a node set $R$ of size $d \le r \le n/2$ such\footnote{The existence of such node set also implies that $x \le n/2$.} that $(R,R^c)$ forms a cut in $H^*_{nmd}$, so that no hyperedge of $H^*_{nmd}$ contains nodes from both $R$ and $R^c$. By the union bound, we find that
\begin{equation}
 \label{eq:Juhannus0}
 \pr(\cD \setminus \cI)
 \wle \sum_{x \le r \le n/2} \binom{n}{r} q_r^m(d),
\end{equation}
where $q_r(d)$ denotes the probability that a set of size $d$ sampled uniformly at random from $[n]$ is fully contained in either $[r]$ or $[n]\setminus [r]$. By \eqref{eq:ShotGunBound} in Lemma \ref{the:CutBoundDet}, we find that 
\[
 q_r(d)
 \wle \left( 1 + \frac{r^2}{(n-r)^2}  \right) (1-d/n)^r
 \wle \exp\left( \frac{r^2}{(n-r)^2}  \right) (1-d/n)^r,
\]
and by noting that $(1-d/n)^r = (n^{-1} e^\lambda)^{r/m}$, we find that
\begin{equation}
 \label{eq:Juhannus1}
 q_r^m(d)
 \wle \left\{ n^{-1} \exp\left( \lambda + \frac{m r}{(n-r)^2} \right) \right\}^r
\end{equation}
for all $d \le r \le n/2$. In addition, \eqref{eq:PairBound} in Lemma \ref{the:CutBoundDet} implies that
$ q_r(d) \le 1 - 2 \frac{r}{n} ( 1 - \frac{r}{n} )$, so that 
\begin{equation}
 \label{eq:Juhannus2}
 q_r^m(d)
 \wle \exp\left( - \frac{2 m r(n-r)}{n^2} \right)
\end{equation}
for all $d \le r \le n/2$. 

Let us now split the index set $J = \{r: d \le r \le n/2\}$ in the sum in \eqref{eq:Juhannus0} into
$J_1 = \{r \in J: r \le n_0\}$ and $J_2 = \{r \in J: r > n_0\}$ using a scale-dependent threshold value $n_0 = (m^{-1}n^2) \wedge (n/2)$, and define
\[
 S_i \weq \sum_{r \in J_i} \binom{n}{r} q_r^m(d), \qquad i=1,2.
\]
By \eqref{eq:Juhannus1} and the inequality $\binom{n}{r} \le (\frac{e n}{r})^r$, we find that 
\begin{align*}
 S_1
 &\wle \sum_{r \in J_1} \left(\frac{e n}{r} \right)^r
   \left\{ n^{-1} \exp\left( \lambda + \frac{m r}{(n-r)^2} \right) \right\}^r  \\
 &\wle \sum_{r =1}^\infty \left\{ \exp\left( \lambda + \frac{m n_0}{(n-n_0)^2} +1 \right) \right\}^r.
\end{align*}
Our choice of $n_0$ now implies that $\frac{m n_0}{(n-n_0)^2} \le \frac{n^2}{(n-n/2)^2} = 4$. Hence,
\begin{equation}
 \label{eq:Juhannus3}
 S_1
 \wle \sum_{r =1}^\infty e^{(\lambda+5)r}.
\end{equation}

We may bound the sum $S_2$ by noting that 
\eqref{eq:Juhannus2} and the equality $\sum_{r=0}^{n} \binom{n}{r} = 2^n$ imply that
\begin{align*}
 S_2
 \wle \sum_{r \in J_2} \binom{n}{r} \exp\left( - \frac{2 m r(n-r)}{n^2} \right)
 \wle 2^n \exp\left( - \frac{2 m n_0(n-n_0)}{n^2} \right).
\end{align*}
Because $n-n_0 \ge n/2$, we conclude that 
\[
 S_2
 \wle \exp\left( n \log 2 - \frac{m n_0}{n} \right).
\]
We also note that $S_2=0$ for $m \le 2n$, because in this case $n_0=n/2$ and the index set $J_2$ is empty. For $m > 2n$ we see that $n_0 = m^{-1}n^2$, and the above inequality becomes
\[
 S_2
 \wle \exp\left( n \log 2 - n \right)
 \weq (2/e)^n.
\]
Therefore, $S_2 \le (2/e)^n$ in both cases.  By recalling that 
$\pr(\cI) \le e^\lambda$ due to \eqref{eq:Juhannus00}, noting that 
$\pr(\cD \setminus \cI) \le S_1+S_2$ due to \eqref{eq:Juhannus0}, and 
combining the latter inequality with \eqref{eq:Juhannus3}, we conclude that
\[
 \pr(\cD)
 \wle \pr(\cI) + S_1 + S_2
 \wle e^\lambda + \sum_{r =1}^\infty e^{(\lambda+5)r} + (2/e)^n.
\]
\end{proof}

\subsection{Connectivity for inhomogeneous layer sizes}

\begin{proposition}
\label{the:GeneralPassiveRIGConnectedNewGen}
If $(F)_2 \lesim (F)_0$ and $\mu = \log n - \frac{m}{n}(F)_1$ satisfies $\mu \to -\infty$,
then $H^*_{nmF}$ is connected whp.
\end{proposition}
\begin{proof}
$H^*_{nmF}$ is disconnected if and only if there exists a node set $A$ of size $1 \le \abs{A} \le n/2$ such that $H^*_{nmF}$ contains no links connecting $A$ to its complement. For any node set with $\abs{A}=r$ nodes, the conditional probability of such event given the hyperedge sizes equals $\prod_{k=1}^m q_r(X_k)$ where $X_k=\abs{V_k}$ and $q_r(x)$ is the probability that a uniformly random $x$-subset of $[n]$ is either fully contained in $[r]$ or fully contained in $[n] \setminus [r]$. Because $X_1,\dots,X_m$ are mutually independent, the corresponding unconditional probability equals $Q_r = \prod_{k=1}^m \E q_r(X_k)$. The union bound hence implies that
\begin{equation}
 \label{eq:DisconnectedRIGDet}
 \pr(\text{$H^*_{nmF}$ disconnected})
 \wle \sum_{1 \le r \le n/2} \binom{n}{r} Q_r.
\end{equation}
We will fix a threshold level $n_0 = c^{-1} m^{-1}n^2$ with
$c = (F)_0$, and split the sum on the right as
\begin{equation}
 \label{eq:SplitGen}
 S
 \weq \underbrace{\sum_{j \in J_1} \binom{n}{r} Q_r}_{S_1}
 \ + \underbrace{\sum_{j \in J_2} \binom{n}{r} Q_r}_{S_2}.
\end{equation}
where
$J_1 = \{1 \le r \le n/2: r \le n_0\}$ and
$J_2 = \{1 \le r \le n/2: r > n_0\}$.

(i) Assume that $n_0 \ge 1$, so that $J_1$ is nonempty, and fix $r \in J_1$. The bound~\eqref{eq:ManyIsolatedBoundNew1} combined with the inequality $\log t \le t-1$ shows that
\[
 \log \E q_r(X_k)
 \wle - \frac{r}{n} \E X_k 1(X_k \ge 2)
 + \frac{r^2}{(n-r)^2} \pr(X_k \ge 2)
 + \frac12 \frac{r^2}{n^2} \E X_k^2 1(X_k \ge 2). 
\]
Recall the moments $(F)_r$ defined by \eqref{eq:FMoment}.
Let $Y$ be a random integer distributed according to probability measure $\frac{1}{m} \sum_{k=1}^m f_k$. By summing both sides of the above inequality,
we find that
\[
 \log Q_r
 \wle - \frac{mr}{n} (F)_1
 + \frac{m r^2}{(n-r)^2} (F)_0
 + \frac12 \frac{m r^2}{n^2} (F)_2.
\]
Because $(F)_2 \ge 4 (F)_0$ and  $(n-r)^2 \ge \frac14 n^2$, 
we see that
\begin{align*}
 \log Q_r
 \wle - \frac{mr}{n} (F)_1 + \frac32 \frac{mr^2}{n^2} (F)_2
 \wle r \left\{ - \frac{m}{n} (F)_1 + \frac32 \frac{m n_0}{n^2} (F)_2 \right\}.
\end{align*}
By applying the bound $\binom{n}{r} \le n^r$, we find that
$
 \binom{n}{r} Q_r
 \le e^{B r},
$
where
\begin{align*}
 B
 \weq \log n - \frac{m}{n} (F)_1 + \frac32 \frac{m n_0}{n^2} (F)_2
 \weq \mu + \frac32 \frac{(F)_2}{(F)_0}.
\end{align*}
Therefore, $S_1 \le \sum_{r=1}^\infty e^{B r}$. Because $\mu \to -\infty$ and $(F)_2 \lesim (F)_0$, we find that $S_1 = o(1)$.

(ii) For the terms of the sum $S_2$, we apply inequality \eqref{eq:ManyIsolatedBoundNew2} and the fact that $r \mapsto \frac{r}{n}( 1 - \frac{r}{n})$ is nondecreasing for $0 \le r \le n/2$, to conclude that
\[
 \E q_r(X_k)
 \wle 1 - c_k + c_k e^{- 2 \frac{r}{n}( 1 - \frac{r}{n}) }
 \wle 1 - c_k + c_k e^{-t}
\]
for all $r \in J_2$, where $c_k = \pr(X_k \ge 2)$ and $t = 2 \frac{n_0}{n}( 1 - \frac{n_0}{n})$.
Because $\log(1+s) \le s$ and $e^{ - t } \le 1 - t + \frac12 t^2$, we find that
$\log ( 1 - c_k + c_k e^{-t} ) \le c_k(e^{-t}-1) \le - c_k t (1 - t/2 ).$
Therefore,
\[
 \log Q_r
 \weq \sum_{k=1}^m \log \E q_r(X_k)
 \wle - \sum_{k=1}^m c_k t (1 - t/2)
 \weq - m c t (1 - t/2).
\]
By applying the equality $\sum_{r=0}^n \binom{n}{r} = 2^n$, it follows that
\begin{align*}
 S_2
 \wle 2^n e^{- m c t (1 - t/2)}
 \weq e^{n \log 2 - m c t (1 - t/2)}.
\end{align*}
Observe now that $(F)_1 \le \frac12 (F)_2$, so that 
$
 \mu
 \ge \log n - \frac{m}{2 n} (F)_2.
$
Because $\mu \to -\infty$ and 
$(F)_2 \lesim (F)_0$
we find that $c = (F)_0 \gg \frac{n}{m}$.
By recalling that $n_0 = c^{-1} m^{-1}n^2$, we see that $n_0 \ll n$. Therefore,
and $m c t (1 - t/2) \sim mc t \sim 2 m c \frac{n_0}{n} \sim 2n$, and we conclude that 
\begin{align*}
 S_2
 \wle e^{n \log 2 - m c t (1 - t/2)}
 \weq e^{n \log 2 - (2+o(1)) n}
 \weq o(1).
\end{align*}
Because $S_1=o(1)$ and $S_2=o(1)$, the claim follows from \eqref{eq:DisconnectedRIGDet}.
\end{proof}

\section{Proofs of main results}
\label{sec:Proofs}

\subsection{Shotgun random graphs}

\begin{proof}[Proof of Theorem~\ref{the:SRH}]

(i) Denote $\mu = \log n - \frac{m}{n} (F)_1$. Assume that $\mu \to -\infty$ and $(F)_2 \lesim (F)_0$.  Then Proposition~\ref{the:GeneralPassiveRIGConnectedNewGen} implies that $H^*_{nmF}$ is connected whp.

(ii) Assume that $\mu \to +\infty$ and $(F)_2 \lesim (F)_0$.  We note that $\mu \to +\infty$ implies $(F)_1 \lesim m^{-1} n \log n$, 
and Proposition~\ref{the:ExpectedIsolatedNewer} implies that $\lambda \ge \mu - O(\frac{\log n}{n})$. Hence $\lambda \to \infty$.
Furthermore, Proposition~\ref{the:YesIsolatedRIG} tells that
\begin{equation}
 \label{eq:FromYesIsolatedRIG}
 \pr( \text{$H^*_{nmF}$ contains no isolated nodes} )
 \wle e^{-\lambda} + \exp\left( \sum_{k=1}^m \frac{\Var(Z_k)}{(\E Z_k)^2} \right) - 1,
\end{equation}
where $Z_k = 1 - n^{-1} X_k 1(X_k \ge 2)$, and $X_k$ is an $f_k$-distributed random integer. 

We will show that $\sum_{k=1}^m \frac{\Var(Z_k)}{(\E Z_k)^2} \to 0$. 
Denote $a = \max_k (f^{(k)})_1$.
We saw in the proof of Proposition~\ref{the:ExpectedIsolatedNewer} that $(F)_2 \lesim (F)_1$ and $a \lesim ( m  (F)_1 )^{1/2} \lesim n^{1/2} \log^{1/2} n$.
As a consequence,
$\min_k \E Z_k = 1 - a/n = 1-o(1)$, and we conclude that
\[
 \sum_{k=1}^m \frac{\Var(Z_k)}{(\E Z_k)^2}
 \wasymp \sum_{k=1}^m \Var(Z_k).
\]

Next, observe that $\Var(Z_k) = n^{-2} \Var(X_k 1(X_k \ge 2)) \le n^{-2} \E X_k^2 1(X_k \ge 2)$,
so it follows that 
\[
 \sum_{k=1}^m \Var(Z_k)
 \wle m n^{-2} (F)_2
 \wlesim m n^{-2} (F)_1
 \wlesim \frac{\log n}{n}.
\]
We now conclude that
$
 \sum_{k=1}^m \frac{\Var(Z_k)}{(\E Z_k)^2}
 \lesim \frac{\log n}{n}
 \to 0.
$
Hence by \eqref{eq:FromYesIsolatedRIG}, together with $\lambda \to \infty$, we conclude $H^*_{nmF}$ contains isolated nodes and is disconnected whp.
\end{proof}

\begin{proof}[Proof of Theorem~\ref{the:SRHRegular}]
Fix integers $m=m_n$ and $d=d_n$ such that $m \ge 1$ and $2 \le d \le n$. Recall that $H^*_{nmd}$ is 
a special instance of the model $H^*_{nmF}$ for which $\lambda$ defined by \eqref{eq:MeanIsolated} reduces to $\lambda = \log n \updated{+} \log(1-d/n)$. Proposition~\ref{the:YesIsolatedRIG} then implies that 
\[
 \pr( \text{$H^*_{nmd}$ is connected} )
 \wle \pr( \text{$H^*_{nmd}$ contains no isolated nodes} )
 \wle e^{-\lambda},
\]
Hence $H^*_{nmd}$ is whp disconnected when $\lambda \to \infty$. Next,
Proposition~\ref{the:ConnectivityConstantSizes} tells that
\[
 \pr(\text{$H^*_{nmd}$ is disconnected})
 \wle e^\lambda + \sum_{r=1}^\infty e^{(\lambda+5)r} + (2/e)^n,
\]
so that $H^*_{nmd}$ is whp connected for $\lambda \to -\infty$.
\end{proof}

\subsection{Random hypergraphs with given hyperedge sizes}

\begin{proof}[Proof of Theorem~\ref{the:Hypergraph}]
Let $H_{nmf}$ be a random hypergraph sampled uniformly at random from the set $\cH_{nmf}$ of hypergraphs on node set $\{1,\dots,n\}$ having $m$ hyperedges and empirical hyperedge size distribution $f$.
Instead of directly analysing $H_{nmf}$, we will study a simpler model generated as follows:
\begin{enumerate}
\item Create a list $(x_1, \dots, x_m)$ by concatenating $mf(1)$ copies of integer~1, $mf(2)$ copies of integer~2, \dots, and $mf(n)$ copies of integer~$n$.
\item Sample random sets $V_1,\dots, V_m$ independently and uniformly at random from the collection of subsets of $[n]$ with sizes $x_1,\dots,x_m$, respectively.
\item Define a hypergraph $H^* = ([n], \{V_1,\dots, V_m\})$.
\end{enumerate}
We note that the distribution of $H^*$ is in general not the uniform distribution on $\cH_{nmf}$, because 
the collection $\{V_1,\dots, V_m\}$ may contain less than $m$ unique sets in case there are duplicates.
Instead, $H^* = H^*_{nmF}$ may be recognised as an instance of the shotgun random hypergraph model defined in Section~\ref{sec:SRH} where $F = \delta_{x_1} \times \cdots \times \delta_{x_m}$ with $\delta_x$ denoting the Dirac point mass at $x$.

Let us denote by $\cD$ the event that the sets $V_1,\dots,V_m$ are all distinct. On this event, $H^* \in \cH_{nmf}$.  We also see\footnote{$H^*=h$ if and only if for each $x$, the $x$-sized hyperedges of $H^*$ coincide with the $x$-sized hyperedges of $h$, and the latter event has probability $m_x! \binom{n}{x}^{-m_x}$ when $h$ has $m_x = mf(x)$ distinct hyperedges of size $x$.} that
\[
 \pr(H^* = h)
 \weq \prod_{x=1}^n (mf(x))! \binom{n}{x}^{-mf(x)}
 \qquad \text{for any $h \in \cH_{nmf}$}.
\]
The above formula shows that the probability mass function $h \mapsto \pr(H^* = h)$ is constant on $\cH_{nmf}$, and therefore the conditional probability distribution on $H^*$ given $\cD$ is the uniform distribution on $\cH_{nmf}$. Especially,
\[
 \pr(\text{$H_{nmf}$ is connected})
 \weq \pr(\text{$H^*_{nmF}$ is connected} \cond \cD).
\]
Furthermore, in this case the special moments $(F)_r$ defined in \eqref{eq:FMoment} coincide with the moments defined in \eqref{eq:fMoment} according to $(F)_r = (f)_r$. Because $\sum_{x=1}^n \binom{n}{x}^{-1} f(x)^2 \ll m^{-2}$, Lemma~\ref{the:Distinct} shows that $\pr( \cD ) \to 1$. The claims now follow by Theorem~\ref{the:SRH}.
\end{proof}

\begin{proof}[Proof of Theorem~\ref{the:HypergraphRegular}]

We will construct random hypergraph $H^*$ in the same way as in the proof of Theorem~\ref{the:Hypergraph}, this time defining the list $(x_1,\dots, x_m)$ simply by concatenating $m$ copies of integer $d$. Then we recognise $H^* = H^*_{nmd}$ as an instance of the shotgun random hypergraph model defined in Section~\ref{sec:SRH} where $F = \delta_{d} \times \cdots \times \delta_{d}$.
Because $m \ll \binom{n}{d}^{1/2}$, Lemma~\ref{the:Distinct} shows that $\pr( \cD ) \to 1$. 
The claims now follow by Theorem~\ref{the:SRHRegular}.
\end{proof}

\subsection{Random intersection graphs}

\begin{proof}[Proof of Theorem~\ref{the:RIG}]
Let $H^*_{nmF}$ be a shotgun random hypergraph with
$F = f \times \cdots \times f$ being the $m$-fold product measure of $f$.
Then the moments defined by \eqref{eq:fMoment} and \eqref{eq:FMoment} match according to
$(F)_r = (f)_r$. The statements of Theorem~\ref{the:RIG} then follow by applying Theorem~\ref{the:SRH} and noting that the graph $G_{nmf}$ is connected if and only if the hypergraph $H^*_{nmf}$ is connected.
\end{proof}

\begin{proof}[Proof of Theorem~\ref{the:RIGRegular}]
Let $H^*_{nmd}$ be a shotgun random hypergraph with
$F$ being the Dirac point mass at $(d,\dots, d) \in \{0,\dots,n\}^m$.  The statements of Theorem~\ref{the:RIGRegular} then follow by applying Theorem~\ref{the:SRHRegular} and noting that the graph $G_{nmd}$ is connected if and only if the hypergraph $H^*_{nmd}$ is connected.
\end{proof}

\appendix
\section{Elementary bounds}

\subsection{Sampling distinct random sets}

Given integers $n,m \ge 1$ and $1 \le x_1,\dots,x_m \le n$, let $V_1,\dots,V_m$ be mutually independent random sets  such that for each $k$, the set $V_k$ is sampled uniformly at random from the collection of all subsets of $\{1,\dots,n\}$ of size $x_k$. The following result (compare with \cite[Theorem 1.1]{Janson_2009}) shows that
\[
 \pr(\text{sets $V_1,\dots,V_m$ are distinct})
 \weq 1-o(1)
\]
if and only if $\sum_{x=1}^n \binom{n}{x}^{-1} \binom{m_x}{2} \ll 1$,
where $m_x$ denotes the number of sets of size~$x$.

\begin{lemma}
\label{the:Distinct}
Denote by $m_x$ the number of sets $V_k$ with size $1 \le x \le n$. The probability $p$ that the sets $V_1, \dots, V_m$ are distinct is nonzero if and only if $m_x \le \binom{n}{x}$ for all $x$. In this case $p$ is bounded by
$1-c \le p \le e^{-c}$ with
$
 c
 = \sum_{x=1}^n \binom{n}{x}^{-1} \binom{m_x}{2}.
$
\end{lemma}
\begin{proof}
The event of interest can be written as $\cD = \cap_{x=1}^n \cD_x$, where
$\cD_x$ is the event that all sets of size $x$ are distinct.  The event $\cD_x$ corresponds to the classical birthday problem (e.g.\ \cite{Mitzenmacher_Upfal_2005} for which we know that
$
 \pr(\cD_x)
 = \prod_{k=1}^{m_x-1} (1 - {k}/{\binom{n}{x}} ).
$
The inequality $1-t \le e^{-t}$ then implies that
\[
 \pr(\cD_x)
 \wle \exp\left( - \sum_{k=1}^{m_x-1} {k}/{\binom{n}{x}} \right)
 \weq \exp\left( - \binom{n}{x}^{-1} \binom{m_x}{2} \right).
\]
Because $\cD_1,\dots,\cD_n$ are independent, it follows that $\pr(\cD) = \prod_{x=1}^n \pr(\cD_x) \le e^{-c}$.

Observe next that when $\cD_x$ fails, then there exists a pair of sets of size $x$ which coincide with each other.  Because any particular pair of sets of size $x$ coincides with probability $\binom{n}{x}^{-1}$, union bound implies that $1-\pr(\cD_x) \le \binom{n}{x}^{-1}\binom{m_x}{2}$. Another union bound then shows that
$1-\pr(\cD) \le \sum_{x=1}^n ( 1-\pr(\cD_x)) \le c$.
\end{proof}

\subsection{Shotgun lemma}
Imagine shooting at a target having $n$ squares, out of which $r$ are painted red, and the rest are white. Let us fire a shotgun with $d \le n-r$ bullets at the target, and assume that bullets hit a uniformly random $d$-subset of the $n$ squares. The probability that none of the bullets hits a red square is
$\binom{n-r}{d} / \binom{n}{d}$.
Imagine a reversed setting, where we fire $r \le n-d$ bullets, trying to avoid $d$ red squares. The probability of success in that case is 
$\binom{n-d}{r} / \binom{n}{r}$. The following lemma confirms that both probabilities are the same.

\begin{lemma}
\label{the:Shotgun}
For any integers $d,r,n \ge 0$ such that $d+r \le n$, the probability
$p = \binom{n-r}{d}/\binom{n}{d}$ can be written as
$p = \binom{n-d}{r}/\binom{n}{r}$
and is bounded by
$p \le (1-r/n)^d$ and $p \le (1-d/n)^r$.
\end{lemma}
\begin{proof}
To verify that $\binom{n-r}{d}/\binom{n}{d} = \binom{n-d}{r} / \binom{n}{r}$, it suffices to write
\[
 \binom{n-r}{d} / \binom{n}{d}
 \weq \frac{(n-r)!}{d!(n-r-d)!} \frac{d!(n-d)!}{n!}
 \weq \frac{(n-r)!}{(n-r-d)!} \frac{(n-d)!}{n!},
\]
and note that the right side remains the same if $d$ and $r$ are swapped.
We also note that
$
 \binom{n-r}{d}/\binom{n}{d}
 = \frac{(n-r)_d}{(n)_d}
 = \prod_{c=0}^{d-1} \frac{n-r-c}{n-c}
 = \prod_{c=0}^{d-1} \left(1 - \frac{r}{n-c}\right),
$
from which we find that $p \le (1-r/n)^d$. The latter inequality follows by repeating the same argument with $d$ and $r$ swapped.
\end{proof}

\subsection{Other bounds}

\begin{lemma}
\label{the:BinomRatio}
For any integers $0 \le d_1 \le d_2 \le n$,
$\binom{n}{d_1} / \binom{n}{d_2} \le (\frac{d_2}{n-d_2+1})^k$ where $k=d_2-d_1$.
\end{lemma}
\begin{proof}
Observe that
\begin{align*}
 \frac{\binom{n}{d_1}}{\binom{n}{d_2}}
 \weq \frac{n!}{d_1! (n-d_1)!} \frac{d_2! (n-d_2)!}{n!} 
 \weq \frac{d_2!}{d_1!} \frac{(n-d_2)!}{(n-d_1)!}.
\end{align*}
Observe that $d_1! (d_1+1)^k \le d_2! \le d_1! d_2^k$,
and similarly, $(n-d_2)! (n-d_2+1)^k \le (n-d_1)! \le (n-d_2)! (n-d_1)^k$. These bounds imply that
\[
 \frac{d_2!}{d_1!}
 \wle d_2^k
 \quad\text{and}\quad
 \frac{(n-d_2)!}{(n-d_1)!}
 \wle \frac{1}{(n-d_2+1)^k},
\]
and the claim follows.
\end{proof}

\begin{lemma}
\label{the:BinomBounds}
For any integers $0 \le k \le n$, $(\frac{n}{k})^k \le \binom{n}{k} \le \frac{n^k}{k!} \le (e \frac{n}{k})^k$.
\end{lemma}
\begin{proof}
We note that
$
 \binom{n}{k}
 = \frac{n}{k} \frac{n-1}{k-1} \cdots \frac{n-k+1}{1}
$
where all terms on the right are at least $\frac{n}{k}$. Hence the first bound follows. The second inequality follows from, we observe that $\binom{n}{k} = \frac{(n)_k}{k!} \le  \frac{n^k}{k!}$. 
The last inequality follow by noting that
$e^k = \sum_{j \ge 0} \frac{k^j}{j!} \ge \frac{k^k}{k!}$.
\end{proof}

\begin{lemma}
\label{the:LogTaylor}
$-t(1-t)^{-1} \le \log(1-t) \le -t$ for all $t \in (0,1)$.
Furthermore, $-t - \frac{t^2}{1-t} \le \log(1-t) \le - t - \frac12 t^2$ for all $t \in (0,1)$.
\end{lemma}
\begin{proof}
Define $f(t) = \log(1-t)$ for $0 < t < 1$. Then $f'(t) = -(1-t)^{-1}$
and
$f(0)=0$. We find that
$
 \log(1-t) = - \int_0^t (1-s)^{-1} ds,
$
from which it follows that
\[
 -t(1-t)^{-1} \wle \log(1-t) \wle -t.
\]
Then $\log(1-t) - (-t) \wge t - \frac{t}{1-t} = -\frac{t^2}{1-t}$, so that $\log(1-t) \ge -t - \frac{t^2}{1-t}$.
We also note that $f''(t) = -(1-t)^{-2}$ and $f'''(t) = -2(1-t)^{-3}$. The formula
$f(t) = f(0) + f'(0)t + \frac12 f''(0)t^2 + \int_0^t \int_0^s \int_0^r f'''(q) \, dq \, dr \, ds$ implies that
$\log(1-t) \le f(0) + f'(0)t \frac12 f''(0) t^2 = -t - \frac12 t^2$.
\end{proof}

\updated{
\subsection*{Acknowledgements}

We thank an anonymous referee for helpful comments that have improved the presentation.
}

\bibliographystyle{siamplain}
\bibliography{lslReferences}

\end{document}